\documentclass[a4paper]{article}
\usepackage{amssymb}
\usepackage{amsmath}
\usepackage{amsthm}
\usepackage[utf8]{inputenc}
\usepackage[T1]{fontenc}
\usepackage[all]{xy}
\usepackage{extarrows}
\usepackage{mathrsfs}
\usepackage{color}
\usepackage{enumitem}
\usepackage{endnotes}
\usepackage{cancel}
\usepackage{mathtools}
\newtheorem{theorem}{Theorem}
\newtheorem{corollary}[theorem]{Corollary}
\newtheorem{proposition}[theorem]{Proposition}
\newtheorem{lemma}[theorem]{Lemma}

\theoremstyle{definition}
\newtheorem{definition}[theorem]{Definition}
\theoremstyle{remark}
\newtheorem{remark}[theorem]{Remark}

\newcommand{\ud}{\mathrm{d}}
\newcommand{\pd}{\partial}
\newcommand{\e}{\varepsilon}
\newcommand{\bC}{\mathbb{C}}
\newcommand{\bN}{\mathbb{N}}
\newcommand{\bR}{\mathbb{R}}
\newcommand{\cA}{\mathcal{A}}
\newcommand{\ccD}{\mathcal{D}}
\newcommand{\cE}{\mathcal{E}}
\newcommand{\cG}{\mathcal{G}}

\newcommand{\cN}{\mathcal{N}}
\newcommand{\csub}{\subset \subset}
\newcommand{\coleq}{\mathrel{\mathop:}=}
\newcommand{\fX}{\mathfrak{X}}
\DeclareMathOperator{\id}{id}
\DeclareMathOperator{\supp}{supp}
 \newcommand{\ocM}{{\Omega^n_c(M)}}
\newcommand{\Cinf}{{C^\infty}}
\newcommand{\Lie}{\mathrm{L}}
\providecommand{\norm}[1]{\left\lVert#1\right\rVert}
\providecommand{\abso}[1]{\left\lvert#1\right\rvert}

\newcommand{\hsk}[2]{{\widetilde\cA_{#1}(#2)}}
\newcommand{\hlsk}[2]{{\widetilde\cA_{#1}(#2)}}
\newcommand{\ball}[2]{{B_{#1}(#2)}}

\newcommand{\D}{\mathrm{D}}
\DeclareMathOperator{\Div}{div}

\newcommand{\hG}[1]{\hat\cG(#1)}
\newcommand{\hE}[1]{\hat\cE(#1)}
\newcommand{\hEm}[1]{\hat\cE_m(#1)}
\newcommand{\hN}[1]{\hat\cN(#1)}
\newcommand{\Dp}[1]{\ccD'(#1)}
\begin{document}

\author{E. A. Nigsch}
\title{A new approach to diffeomorphism invariant algebras of generalized functions}
\maketitle



\begin{abstract}
We develop the diffeomorphism invariant Colombeau-type algebra of nonlinear generalized functions in a modern and compact way. Using a unifying formalism for the local setting and on manifolds, the construction becomes simpler and more accessible than previously in the literature.
\end{abstract}

\section{Introduction}

In the 1980s J.~F.~Colombeau introduced algebras of nonlinear generalized functions (\cite{colnew, colelem}) in order to overcome the long-standing problem of multiplying distributions, retaining as much compatibility with the classical theory as possible in light of the Schwartz impossibility result (\cite{Schwartz}). These algebras and later variations, nowadays simply known as \emph{Colombeau algebras}, contain the algebra of smooth functions as a faithful subalgebra and the vector space of Schwartz distributions as a linear subspace (see \cite{MOBook, GKOS} for a comprehensive survey).

A diffeomorphism invariant formulation of the theory was first proposed by Colombeau and Meril in \cite{colmani}, but later seen to be flawed by J.~Jel\'inek who presented a new version in \cite{jelinek}, which was subsequently refined in \cite{found}. The difficulties inherent in this development stem from the combination of three facets (see~\cite[Chapter 2]{GKOS} for a detailed discussion): first, one needs to employ a suitable notion of calculus on (non-Fr\'echet) locally convex spaces. Second, the proper handling of diffeomorphism invariance manifestly presents a major hurdle in the constructions cited above, both conceptually and technically. And third, establishing stability of the algebra under differentiation is far from trivial and requires a delicate treatment. For this reason the published results in this area consist of several long, technically involved papers which are difficult to assimilate for those not already working in the field.

In this article we give a systematically refined presentation of the global theory of full Colombeau algebras, based on the algebras $\cG^d$ of \cite{found} and $\hat\cG$ of \cite{global} but replacing a significant part of the preceding foundational material by a succint, more efficient approach.

Our presentation is based, both locally and on manifolds, on the formalism of \cite{global}, where so-called smoothing kernels are used as key components of the construction. This not only simplifies the local case in several respects compared to \cite{found} but also makes the translation to manifolds much more convenient.
En passant, several proofs of \cite{found} were simplified; in particular, we give a significantly shorter proof of stability under differentiation.
Finally, we establish the few core properties of smoothing kernels on which the whole theory depends separately, which makes for a clearer and less technical presentation.

\section{Preliminaries}

$\ball{r}{x}$ denotes the open ball of radius $r > 0$ centered at $x \in \bR^n$ with respect to the Euclidean metric. $\pd_i$ denotes the $i$th partial derivative; we employ common multi--index notation where for $\alpha=(\alpha_1,\dotsc,\alpha_n) \in \bN_0^n$ we have $\pd^\alpha = \pd_1^{\alpha_1} \dotsm \pd_n^{\alpha_n}$. $\pd_x^\alpha$ means the derivative in the $x$-variable. We abbreviate $\pd_{x+y}^\alpha \coleq (\pd_x + \pd_y)^\alpha$ which gets expanded by the binomial theorem, $(-\pd_x)^\alpha \coleq (-1)^{\abso{\alpha}} \pd_x^\alpha$ and $\pd_{(x,y)}^{(\alpha,\beta)} = \pd_x^\alpha\pd_y^\beta$. $\D_X$ means the directional derivative on functions with respect to a vector field $X$, with $\D_X^x$ denoting the directional derivative in the variable $x$. $\{e_1,\dotsc,e_n\}$ is the Euclidean basis of $\bR^n$.

We use the Landau notation $f(\e) = O(g(\e))$ for $\exists \e_0>0,C>0$: $\abso{f(\e)} \le C g(\e)$ $\forall \e \le \e_0$. $\ccD(\Omega)$ and $\ccD'(\Omega)$ denote the space of test functions and distributions on $\Omega$, respectively. The action of a distribution $u$ on a test function $\varphi$ is written as $\langle u, \varphi \rangle$. Given open subsets $\Omega, \Omega'$ of $\bR^n$, the pullback $\mu^*\rho$ of $\rho \in \ccD(\Omega')$ along a diffeomorphism $\mu\colon \Omega \to \Omega'$ is the element of $\ccD(\Omega)$ given by $(\mu^* \rho)(y) \coleq \rho(\mu y) \cdot \abso{\det \D\mu(y)}$, where $\D\mu(y)$ is the Jacobian of $\mu$ at $y$ and $\mu_* \coleq (\mu^{-1})^*$. Accordingly, $\Lie_X \varphi=\ud / \ud t|_{t=0} ((\alpha_t)^* \varphi)$ equals $\D_X \varphi + \Div X \cdot \varphi$, where $\alpha_t$ is the flow of $X$ at time $t$ and $\Div X = \sum_i \pd X^i  / \pd x_i$. The Lie derivative of a distribution $u$ along $X$ is then given by $\langle \Lie_X u, \varphi \rangle  = - \langle u, \Lie_X \varphi \rangle$.

A manifold will always mean an orientable smooth paracompact Hausdorff manifold of finite dimension. The space of distributions on a manifold $M$ is given by $\ccD'(M) \coleq \Omega^n_c(M)'$, where $\Omega^n(M)$ is the space of $n$-forms on $M$ and $\Omega^n_c(M)$ the subspace of those with compact support. We refer to \cite[Section 3.1]{GKOS} for a comprehensive exposition of distributions on manifolds. The Lie derivative of functions and $n$-forms on a manifold w.r.t.~a vector field $X$ is denoted $\Lie_X$ with $\Lie_X^x$ explicitly denoting the derivative in the $x$-variable. $\fX(M)$ is the space of smooth vector fields on $M$ and $B_r^h(x)$ is the ball of radius $r$ centered at $x$ with respect to a Riemannian metric $h$.

$A \csub B$ means that $A$ is compact and contained in the interior of $B$. We set $I \coleq (0,1]$. Calculus of smooth functions on infinite-dimensional locally convex vector spaces is to be understood in the sense of convenient calculus of \cite{KM}, whose basics are presumed to be known. In particular, we use the differentiation operator $\ud$, the fact that linear bounded maps are smooth, and that the notion of smoothness in convenient calculus agrees with the classical one for finite-dimensional spaces. For a multivariate function $f$, $\ud_i f$ means the differential in the $i$th variable.

Finally, we refer to \cite{Dowker} for notions of sheaf theory.

\section{Construction of the algebra}\label{sec_core}

We recall the steps in the construction of a Colombeau algebra on an open set $\Omega \subseteq \bR^n$. One starts with the \emph{basic space} $\hE\Omega$, which contains the representatives of generalized functions, together with embeddings of smooth functions and distributions. The action of diffeomorphisms and derivatives on the basic space is then given, extending their classical counterparts. Next follows the definition of \emph{test objects}, which are used to define the subalgebra $\hEm\Omega \subseteq \hE\Omega$ of \emph{moderate} functions and the ideal $\hN\Omega$ of \emph{negligible} functions. This in turn gives rise to the \emph{quotient algebra} $\hG\Omega$. One then verifies the desired properties of the embeddings, the sheaf property and the invariance of negligibility and moderateness under differentiation, which makes the construction complete.

\begin{definition}\label{basedef}
\begin{enumerate}[label=(\roman*)]
 \item The basic space is $\hE \Omega \coleq \Cinf(\ccD(\Omega) \times \Omega)$, the space of all smooth functions $R: (\varphi, x) \mapsto R(\varphi, x)$ on the product space $\ccD(\Omega) \times \Omega$. The embeddings $\iota\colon \Dp \Omega \to \hE \Omega$ and $\sigma\colon \Cinf(\Omega) \to \hE \Omega$ are defined as $(\iota u)(\varphi, x) = \langle u, \varphi \rangle$ for a distribution $u$ and $(\sigma f)(\varphi, x) = f(x)$ for a smooth function $f$, where $\varphi \in \ccD(\Omega)$ and $x \in \Omega$.
\item Let $\mu\colon \Omega \to \Omega'$ be a diffeomorphism onto another open subset $\Omega'$ of $\bR^n$. Given a generalized function $R \in \hE{\Omega'}$, its pullback $\mu^*R \in \hE\Omega$ is defined as $(\mu^* R)(\varphi, x) = R(\mu_*\varphi, \mu x)$.
\item The derivative of $R \in \hE\Omega$ with respect to a vector field $X \in \Cinf(\Omega, \bR^n)$ is defined as $(\hat \Lie_X R)(\varphi, x) = - \ud_1 R(\varphi, x)(\Lie_X\varphi) + (\D_X^x R)(\varphi, x)$.
\end{enumerate}
\end{definition}

\begin{remark}\label{diffcommute}
\begin{enumerate}[label=(\roman*)]
 \item The formula for $\hat\Lie_X$ is obtained by considering the pullback of $R$ along the flow of a (complete) vector field and taking its derivative at time zero.
 \item One has to verify that $\iota$, $\sigma$, $\mu^*$ and $\hat \Lie_X$ actually map into $\hG M$. First, $\iota u\colon (\varphi, x) \mapsto \varphi \mapsto \langle u, \varphi \rangle$ is smooth because continuous linear functions are smooth. Second, $\sigma f\colon (\varphi, x) \mapsto x \mapsto f(x)$ is smooth because $f$ is. Third, $\mu^*\colon \ccD(\Omega') \to \ccD(\Omega)$ as well as $\Lie_X\colon \ccD(\Omega) \to \ccD(\Omega)$ are linear and continuous and thus smooth, which implies the same for their extension to $\hE\Omega$.
\item $\hE\Omega$ is an associative commutative algebra with unit $\sigma(1)\colon(\varphi, x) \mapsto 1$, $\iota$ is a linear embedding and $\sigma$ an algebra embedding. From the definition one sees that pullback and directional derivatives commute with the embeddings.
\item $\hat\Lie_X$ is only $\bR$-linear but not $\Cinf(\Omega)$-linear in $X$; because it commutes with $\iota$, the latter property would in fact give a contradiction similar to the Schwartz impossibility result.
\end{enumerate}
\end{remark}

%

For the quotient construction we employ spaces of smoothing kernels $\hlsk q \Omega$. We give their definition and additional properties now but postpone proofs until Section \ref{sec_testobj} in order to separate the definitions and main theorems of the theory from the more intricate and technically involved details.
\begin{definition}\label{skdef}A \emph{smoothing kernel of order $q \in \bN_0$} on an open subset $\Omega$ of $\bR^n$ is a mapping $\tilde \phi \in \Cinf(I \times \Omega, \ccD(\Omega))$, $(\e,x) \mapsto [ y \mapsto \tilde \phi_{\e,x}(y)]$, satisfying the following conditions:
\begin{enumerate}
 \item[(LSK1)] $\forall K \csub \Omega$ $\exists \e_0,C>0$ $\forall x \in K$ $\forall \e<\e_0$: $\supp \tilde\phi_{\e,x} \subseteq \ball{C\e}x$, 
 \item[(LSK2)] $\forall K \csub \Omega$ $\forall \alpha,\beta \in \bN_0^n$: $(\pd_{x+y}^\alpha \pd_{y}^\beta \tilde \phi)_{\e,x}(y) = O(\e^{-n-\abso{\beta}})$ uniformly for $x \in K$ and $y \in \Omega$,
 \item[(LSK3)] $\forall K \csub \Omega$ $\forall \alpha \in \bN_0^n$ $\forall f \in \Cinf(\Omega)$:
$\int_\Omega f(y) (\pd_x^\alpha \tilde \phi)_{\e,x}(y) \,\ud y = (\pd^\alpha f)(x) + O(\e^{q+1})$
uniformly for $x \in K$.
\end{enumerate}
The space of all smoothing kernels of order $q$ on $\Omega$ is denoted by $\hlsk q \Omega$ and is an affine subspace of $\Cinf(I \times \Omega, \ccD(\Omega))$. The linear subspace parallel to it, denoted by $\hlsk {q0} \Omega$, is given by all $\tilde \phi$ satisfying (LSK1), (LSK2) and the following condition:
\begin{enumerate}
 \item[(LSK3')] $\forall K \csub \Omega$ $\forall \alpha  \in \bN_0^n$ $\forall f \in \Cinf(\Omega)$: $\int_\Omega f(y) (\pd_x^\alpha \tilde \phi)_{\e,x}(y) \,\ud y = O(\e^{q+1})$
uniformly for $x \in K$.
\end{enumerate}
\end{definition}

\begin{remark}\label{remark}
Given $\tilde \phi$ in $\hlsk q \Omega$ or $\hlsk {q0} \Omega$ and a vector field $X \in \Cinf(\Omega, \bR^n)$, $(\D_X^x + \Lie_X^y)\tilde \phi$ is an element of $\hlsk {q0} \Omega$. In fact, $((\D_X^x + \Lie_X^y)\tilde \phi)_{\e,x} = (\D_X^{x+y}\tilde \phi)_{\e,x} + \Div X \cdot \tilde \phi_{\e,x}$. For (LSK1), let $K \csub \Omega$ and choose $L$ with $K \csub L \csub \Omega$. Then for some $C>0$ such small $\e$, $\supp \tilde \phi_{\e,x} \subseteq \ball {C \e}x$ $\forall x \in L$, which implies the same for $(\D_X^x\tilde \phi)_{\e,x}$ and $(\D_X^y \tilde \phi)_{\e,x}$ if $x \in K$. For (LSK2) we note that with $X= (X^1,\dotsc,X^n)$, $(\D_X^{x+y}\tilde \phi)_{\e,x}(y)$ equals $\sum_i ((X^i(x)\pd_{x_i + y_i} + (X^i(y) - X^i(x))\pd_{y_i}) \tilde \phi)_{\e,x}(y)$; the first term of each summand can be estimated by $O(\e^{-n})$ and the second by
\[ \sup_{y \in \ball{C\e}x} \abso{X^i(y) - X^i(x)} \cdot \sup_{y \in \Omega} \abso{\pd_{y_i}\tilde \phi_{\e,x}(y)} = O(\e) O(\e^{-n-1}) = O(\e^{-n}) \]
for some $C>0$ uniformly for $x$ in compact sets, and similarly for its derivatives. (LSK3') is clear from the definitions.
\end{remark}


\begin{definition}Let $\mu\colon \Omega \to \Omega'$ be a diffeomorphism. We define the pullback $\mu^*\tilde \phi$ of a smoothing kernel $\tilde \phi \in \hlsk q {\Omega'}$ by $(\mu^*\tilde \phi)_{\e,x}(y) \coleq \mu^*(\tilde \phi_{\e, \mu x})(y) = \tilde \phi_{\e, \mu x}(\mu y) \cdot \abso{\det \D\mu(y)}$.
\end{definition}

By smoothness of $\mu$ and $\mu^*\colon \ccD(\Omega') \to \ccD(\Omega)$, $\mu^*\tilde \phi = \mu^* \circ \tilde \phi \circ (\id \times \mu)$ is an element of $C^\infty(I \times \Omega, \ccD(\Omega))$, where $\id$ is the identity mapping.

\begin{proposition}\label{additional}The smoothing kernels of Definition \ref{skdef} satisfy these properties:
 \begin{enumerate}
\item[(LSK4)]\label{lsk2} Let $U,V$ be open subsets of $\Omega$, $K \csub U \cap V$ and $q \in \bN_0$. Given $\tilde \phi \in \hlsk q U$ there exist $\e_0>0$ and $\tilde \psi \in \hlsk q V$ such that $\tilde \phi_{\e,x} = \tilde \psi_{\e,x}$ for $\e < \e_0$ and $x \in K$.
 \item[(LSK5)]\label{lsk5} $\forall u \in \Dp\Omega$ $\forall \tilde\phi \in \hlsk 0 \Omega$ $\forall k \in \bN_0$ $\forall X_1,\dotsc,X_k \in C^\infty(\Omega, \bR^n)$: $\langle u, \D_{X_1}^x \dotsm \D_{X_k}^x \tilde\phi_{\e,x}\rangle$ converges to $\Lie_{X_1} \dotsc \Lie_{X_k} u$ in $\ccD'(\Omega)$ for $\e \to 0$.
\item[(LSK6)]\label{lsk6} Given a diffeomorphism $\mu\colon \Omega \to \Omega'$ and $\tilde\phi \in \hlsk q {\Omega'}$, $\mu^*\tilde \phi \in \hlsk q \Omega$.
\item[(LSK7)]\label{lsk7} Given $\tilde \phi_0 \in \hlsk q \Omega$, $\delta \in \bN_0^n$, $\tilde \phi_{\beta} \in \hlsk {q0} \Omega$ for all $\beta \ne 0$, $\beta \le \delta$, a sequence $(\e_j)_{j \in \bN}$ with $0 < \e_{j+1} < \e_j < 1/j$ $\forall j \in \bN$, a sequence $(x_j)_{j \in \bN}$ in a set $K \csub \Omega$ and functions $\lambda_j$ as in Lemma \ref{lambdalemma}, the function $\tilde \psi \in \Cinf(I \times \Omega, \ccD(\Omega)$ defined by
\[
      \tilde \psi_{\e,x}(y) \coleq \sum_{j=1}^{\infty} \lambda_j(\e) \left(\frac{\e_j}{\e}\right)^n \sum_{\beta \le \delta} \frac{(x-x_j)^{\beta}}{\beta!} (\tilde\phi_{\beta})_{\e_j, x_j} \left(\e_j \frac{y-x}{\e} + x_j\right)
\]
is an element of $\hlsk q {\bR^n}$.
\end{enumerate}
\end{proposition}

\begin{remark}(LSK4) is of value in several proofs, essentially stating that during testing smoothing kernels can be restricted and extended as needed. In (LSK5) one can equivalently demand that $\langle u, (\D_{X_1}^x + \Lie_{X_1}^y) \dotsm (\D_{X_1}^x + \Lie_{X_k}^y) \tilde \phi_{\e,x} \rangle$ converges to $0$ for $k>0$ and to zero for $k=0$. (LSK7) gives smoothing kernels taking prescribed values at chosen points and is needed to prove stability of moderateness and negligibility under directional derivatives.
\end{remark}

We can now formulate the definitions of moderateness and negligibility.

\begin{definition}\label{locmodneg}\begin{enumerate}
\item[(i)] $R \in \hE\Omega$ is called \emph{moderate} if $\forall K \csub \Omega$ $\forall \alpha \in \bN_0^n$ $\exists q \in \bN_0$ $\exists N \in \bN$ $\forall \tilde\phi \in \hlsk q\Omega$: $\sup_{x \in K} \abso{\pd_x^\alpha (R(\tilde\phi_{\e,x}, x))} = O(\e^{-N})$. The set of all moderate elements of $\hE\Omega$ is denoted by $\hEm\Omega$.
 \item[(ii)] $R \in \hE\Omega$ is called \emph{negligible} if $\forall K \csub \Omega$ $\forall \alpha\in\bN_0^n$ $\forall m \in \bN$ $\exists q \in \bN_0$ $\forall \tilde\phi\in \hlsk q \Omega$: $\sup_{x \in K} \abso{\pd_x^\alpha (R(\tilde\phi_{\e,x}, x))} = O(\e^m)$. The set of all negligible elements of $\hE\Omega$ is denoted by $\hN\Omega$.
\end{enumerate}
\end{definition}

\begin{remark}\label{technical}In the original definition of $\cG^d$ the moderateness test (translated to the formalism using smoothing kernels) had to be satisfied for \emph{all} $\tilde \phi \in \hlsk 0\Omega$; because this produces a purely technical artefact in the definition of point values and manifold-valued functions (\cite{mfval,punktwerte}) we prefer the test with $\tilde \phi \in \hlsk q\Omega$ for some $q$, where this does not appear. And what's more, this gives in fact an isomorphic algebra, as has been shown in \cite{Jelinek2}. Furthermore, we have stronger conditions on the smoothing kernels than \cite{found}, which only requires $\alpha=0$ in (LSK3), but the resulting algebras are again isomorphic (\cite[Corollary 16.8]{found}).
\end{remark}

As in other variants of the theory the negligibility test is simplified if the tested function is already known to be moderate; the proof uses the same argument as in all the other variants of the theory (\cite[Theorem 1.2.3]{GKOS}).

\begin{proposition}$R \in \hEm\Omega$ is negligible if and only if Definition \ref{locmodneg} (ii) holds for $\alpha=0$, i.e., 
 $\forall K \csub \Omega$ $\forall m \in \bN$ $\exists q \in \bN_0$ $\forall \tilde\phi\in \hlsk q \Omega$: $\sup_{x \in K} \abso{R(\tilde\phi_{\e,x}, x)} = O(\e^m)$.
\end{proposition}
\begin{proof}
 Suppose $R$ satisfies Definition \ref{locmodneg} (ii) for $\alpha=\alpha_0 \in \bN_0^n$ and fix sets $K_0 \csub L \csub \Omega$, $m_0 \in \bN$ and $1 \le i \le n$. Testing $R$ for moderateness on $L$ with $\alpha=\alpha_0 + 2 e_i$ gives $q_1 \in \bN_0$ and $N \in \bN$. By assumption the negligibility test on $L$ with $\alpha=\alpha_0$ and $m = 2m_0+N$ gives some $q_2 \in \bN_0$. Take $q = \max(q_1, q_2)$ and $\tilde\phi \in \hlsk q \Omega$. Define $f_\e \in \Cinf(\Omega)$ by $f_\e(x) = \pd_x^{\alpha_0}(R(\tilde\phi_{\e,x},x))$. Then for small $\e$, $x + [0,1] \cdot \e^{m_0+N}e_i \subseteq L$ for all $x \in K_0$, so $f_\e(x + \e^{m_0+N}e_i) = f_\e(x) + (\pd_{x_i} f_\e)(x)\e^{m_0+N} + \int_0^1(1-t)(\pd_i^2 f_\e)(x + t\e^{m_0 + N}e_i)\e^{2m_0 + 2N}\,\ud t$. Then $(\pd_{x_i} f_\e)(x)$ is given by $(f_\e(x + \e^{m_0+N}e_i) - f_\e(x)) \cdot \e^{-m_0-N} - \int_0^1(1-t)(\pd_i^2 f_\e)(x + t\e^{m_0 + N}e_i)\e^{m_0 + N}\,\ud t = O(\e^{m_0})$ uniformly for $x \in K_0$, which shows that $R$ satisfies the negligibility test on $K_0$ for $\alpha=\alpha_0 + e_i$ and $m=m_0$. By induction $R$ is negligible.
\end{proof}

\begin{theorem}
(i) $\iota(\Dp\Omega) \subseteq \hEm\Omega$, (ii) $\sigma(\Cinf(\Omega)) \subseteq \hEm\Omega$, (iii) $(\iota - \sigma)(\Cinf(\Omega)) \subseteq \hN\Omega$, (iv) $\iota(\Dp\Omega) \cap \hN\Omega = \{ 0 \}$.
\end{theorem}

\begin{proof}
(i) Let $u \in \Dp\Omega$ be given. Fix $K \csub L \csub \Omega$, $\alpha \in \bN_0^n$ and set $q=0$. Given $\tilde\phi \in \hlsk q \Omega$ the moderateness test involves estimating $\pd_x^\alpha((\iota u)(\tilde\phi_{\e,x}, x)) = \pd_x^\alpha \langle u, \tilde\phi_{\e,x} \rangle = \langle u, \pd_x^\alpha \tilde\phi_{\e,x} \rangle$ for $x \in K$. By (LSK1) $\tilde\phi_{\e,x}$ and its derivatives 
have support in $L$ for small $\e$ and $x \in K$, so by the usual seminorm estimate for distributions
and (LSK2) there exist some $C>0$ and $m \in \bN$ depending only on $u$ and $L$ such that this expression can be estimated by $C \sup_{\abso{\beta}\le m,x \in K,y \in L} \abso{ \pd_y^\beta \pd_x^\alpha\tilde\phi_{\e,x}(y) } = O(\e^{-n-\abso{\alpha}-\abso{\beta}})$.

(ii) is clear because derivatives of $f \in \Cinf(\Omega)$ are bounded on compact sets independently of $\e$.

(iii) For $K \csub \Omega$, $\alpha \in \bN_0^n$, $f \in \Cinf(\Omega)$ and $m \in \bN$ we have for all $\tilde \phi \in \hlsk {m-1}\Omega$ that $\pd_x^\alpha ((\iota f)(\tilde\phi_{\e,x},x)) = \langle f, (\pd_x^\alpha \tilde\phi)_{\e,x} \rangle = (\pd^\alpha f)(x) + O(\e^m) = \pd_x^\alpha ((\sigma f)(\tilde \phi_{\e,x}, x)) + O(\e^m)$ uniformly for $x \in K$ by (LSK3).

(iv) Let $u \in \Dp\Omega$ with $\iota u \in \hN\Omega$ and $\varphi \in \ccD(\Omega)$. Then with $\tilde \phi \in \hlsk q \Omega$ for some $q$ the function in $x$ given by $\langle u, \tilde \phi_{\e,x} \rangle$ converges to $0$ uniformly for $x \in \supp \varphi$ when $\e \to 0$ because of negligibility of $\iota u$, thus $\langle \langle u, \tilde \phi_{\e,x}\rangle, \varphi(x) \rangle$ converges to $0$. On the other hand, by (LSK5) $\langle u, \tilde \phi_{\e,x} \rangle$ converges to $u$ in $\ccD'(\Omega)$, which implies $u=0$.
\end{proof}

The following is easily verified with the respective definitions.

\begin{theorem}$\hEm\Omega$ is a subalgebra of $\hE\Omega$ and $\hN\Omega$ is an ideal in $\hEm\Omega$.
\end{theorem}
%
We can now define the algebra of generalized functions on $\Omega$ (isomorphic to $\cG^d(\Omega)$ of \cite{found}) as the quotient of moderate modulo negligible functions.
\begin{definition}$\hG\Omega \coleq \hEm\Omega / \hN\Omega$.
\end{definition}

Diffeomorphism invariance of $\hat\cG$ now follows from (LSK6).

\begin{proposition}\label{diffinv}Let $\mu\colon \Omega \to \Omega'$ be a diffeomorphism. Then $\mu^*(\hEm{\Omega'}) \subseteq \hEm\Omega$ and $\mu^*(\hN{\Omega'}) \subseteq \hN\Omega$, thus $\mu$ is well-defined on $\hat\cG$ by its action on representatives.
\end{proposition}

From Remark 2 (iii) it now follows that $\iota$ and $\sigma$, considered as maps into $\hat\cG(\Omega)$, also commute with diffeomorphisms.

\section{Sheaf properties}

\begin{definition}
 Let $R \in \hE\Omega$ and $\Omega' \subseteq \Omega$ open. Then the restriction $R|_{\Omega'} \in \hE{\Omega'}$ is defined as $R|_{\Omega'}(\omega, x) \coleq R(\omega, x)$ for $\omega \in \ccD(\Omega') \subseteq \ccD(\Omega)$ and $x \in \Omega'$.
\end{definition}
Employing (LSK4) one immediately obtains that moderateness and negligibility are local properties, which makes restriction well-defined also on the quotient space:

\begin{proposition}\label{localize}
\begin{enumerate}
 \item[(i)]Let $\Omega' \subseteq \Omega$ be open and $R \in \hE\Omega$. If $R$ is moderate or negligible, respectively, then so is $R|_{\Omega'}$.
 \item[(ii)] Let $(U_\alpha)_\alpha$ be an open covering of $\Omega$ and $R \in \hE\Omega$. If for all $\alpha$, $R|_{U_\alpha}$ is moderate or negligible, respectively, then so is $R$.
\end{enumerate}
\end{proposition}

\begin{definition}
Let $\hat T \in \hG\Omega$ and $\Omega' \subseteq \Omega$. Then the restriction $\hat T|_{\Omega'} \in \hG{\Omega'}$ of $\hat T$ to $\Omega'$ is defined as $\hat T|_{\Omega'} \coleq T|_{\Omega'} + \hN{\Omega'}$ where $T \in \hEm\Omega$ is any representative of $\hat T$.
\end{definition}

\begin{proposition}\label{itsasheaf}$\hat\cG$ is a fine sheaf of differential algebras.
\end{proposition}
\begin{proof}
Let $U \subseteq \bR^n$ be open and $\{U_\lambda\}_\lambda$ an open cover of $U$. Suppose that for each $\lambda$ we are given an element $\hat T_\lambda \in \hG{U_\lambda}$ represented by $T_\lambda \in \hEm{U_\lambda}$ such that $(\hat T_\lambda - \hat T_\mu)|_{U_\lambda \cap U_\mu}$ is zero for all $\lambda$ and $\mu$. We have to show that there exists a generalized function $\hat T \in \hG U$ such that $\hat T|_{U_\lambda} = \hat T_\lambda$ for all $\lambda$.  By Proposition \ref{localize} (ii), $\hat T$ then is unique with this property.

Let $\{\chi_j\}_j$ be a locally finite partition of unity such that each $\chi_j$ has compact support in $U_{\lambda(j)}$ for some $\lambda(j)$. For each $j$ choose an open neighborhood $W_j$ of $\supp \chi_j$ which is relatively compact in $U_{\lambda(j)}$ and a function $\theta_j \in \ccD(U_{\lambda(j)})$ which is $1$ on $\overline{W_j}$. Define $\pi_j \in \Cinf(\ccD(U), \ccD(U_{\lambda(j)}))$ by $\pi_j(\omega) \coleq \theta_j \cdot \omega$ for all $j$ and $T \in \Cinf(\ccD(U) \times U)$ by $T(\omega, x) \coleq \sum_{j} \chi_j(x) \cdot T_{\lambda(j)}(\pi_j(\omega), x)$. Because the family $\{W_j\}_j$ and thus also $\{\supp \chi_j\}_j$ are locally finite this sum is well-defined and smooth.

Fix $K \csub U$ and $\alpha \in \bN_0^n$ for the moderateness test. Because $K$ has an open neighborhood intersecting only finitely many $\supp \chi_j$ there is a finite set $F$ such that for all $\tilde \phi \in \hlsk 0 U$, $\alpha \in \bN_0^n$ and $x \in K$, $\pd_x^\alpha (T(\tilde \phi_{\e,x}, x)) = \sum_{j \in F}\pd_x^\alpha (\chi_j(x) \cdot T_{\lambda(j)}(\pi_j(\tilde \phi_{\e,x}), x))$. For $T$ to be moderate it therefore suffices to show that for each fixed $j \in F$, any $L \csub W_j$ and any $\beta \in \bN_0^n$ there exist $q \in \bN_0$ and $N \in \bN$ such that if $\tilde \phi$ is of order $q$ then $\pd_x^\beta (T_{\lambda(j)}(\pi_j(\tilde \phi_{\e,x}), x)) = O(\e^{-N})$ uniformly for $x$ in $L$.

Fixing $j$, $L$ and $\beta$ there are $q$ and $N$ such that for all $\tilde \psi \in \hlsk q {U_{\lambda(j)}}$ we have $\pd_x^\beta (T_{\lambda(j)}(\tilde \psi_{\e,x},x)) = O(\e^{-N})$ uniformly for $x \in L$. In particular, given $\tilde \phi \in \hlsk q U$ let $\tilde \psi$ be determined by (LSK4) such that $\tilde \psi_{\e, x} = \tilde \phi_{\e,x}$ for small $\e$ and $x$ in an open neighborhood of $L$ whose closure is compact and contained in $W_j$. By (LSK1) then for small $\e$, $\supp \tilde \phi_{\e,x} \subseteq W_j$ for all $x$ in this neighborhood and hence $\pd_x^\beta (T_{\lambda(j)}(\pi_j(\tilde \phi_{\e,x}), x)) = \pd_x^\beta(T_{\lambda(j)}(\tilde \psi_{\e,x}, x))$ for $x \in L$, which implies moderateness of $T$.

Set $\hat T = T + \hN U$. For $\hat T|_{U_\lambda} = \hat T_\lambda$ it suffices by assumption, Proposition \ref{localize} (ii) and because $\{\,W_k\,\}_k$ is an open cover of $U$, to show negligibility of $T|_{U_\lambda \cap W_k} - T_{\lambda(k)}|_{U_\lambda \cap W_k}$ for all $k$. Because $U_\lambda \cap W_k$ is relatively compact there is a finite set $F$ such that $(T - T_{\lambda(k)})|_{U_\lambda \cap W_k}(\omega, x)$ is given by $\sum_{j \in F} \chi_j(x) (T_{\lambda(j)}(\pi_j(\omega), x) - T_{\lambda(k)}(\omega, x))$ on its domain of definition. For testing a single summand for negligibility fix $j \in F$, $K \csub U_\lambda \cap W_k$ and $m \in \bN$. By assumption there exist $q$ and $N$ such that for all $\tilde \psi \in \hlsk q {U_{\lambda(j)} \cap U_{\lambda(k)}}$, $(T_{\lambda(j)} - T_{\lambda(k)})(\tilde\psi_{\e,x}, x) = O(\e^m)$ uniformly for $x \in K \cap \supp \chi_j$. In particular, given $\tilde \phi \in \hlsk q {U_\lambda \cap W_k}$ let $\tilde \psi$ be determined by (LSK4) such that $\tilde \psi_{\e,x} = \tilde \phi_{\e,x}$ for $x \in K \cap \supp \chi_j$ and small $\e$. By (LSK1), the support of $\tilde \phi_{\e,x}$ is contained in $W_j$ for all $x \in K\cap \supp \chi_j$ and small $\e$. This implies $T_{\lambda(j)}(\pi_j(\tilde \phi_{\e,x}), x) = T_{\lambda(j)}(\tilde \psi_{\e,x}, x)$, giving the desired estimate.

That $\hat\cG$ is \emph{fine} sheaf may be inferred from the fact that it is a sheaf of modules over the soft sheaf $C^\infty$ (\cite[Theorem 9.16]{Bredon}).
\end{proof}

\section{Stability under differentiation}

\begin{theorem}\label{diffeoinv}
 Let $R \in \hE\Omega$ and $X \in \Cinf(\Omega, \bR^n)$. Then (a) $R \in \hEm\Omega$ implies $\hat\Lie_X R \in \hEm\Omega$, and (b) $R \in \hN\Omega$ implies $\hat\Lie_X R \in \hN\Omega$.
\end{theorem}
\begin{proof}
If $X=e_i$ for some $i \in \{1,\dotsc,n\}$ set $\kappa\coleq 0$, otherwise assume the result holds for $X=e_i$ for some $i$ and set $\kappa\coleq 1$. This means that this proof has to be read twice --- both cases follow the same scheme, but the second requires the first as a prerequisite. Let $\mu\colon (t,x) \mapsto \mu_t x$ be the flow of $X$. The claim follows from estimates of $\pd_x^\alpha(\pd_t - \kappa \D_X^x)(R(\mu_{-t}^*\tilde \phi_{\e,x}, \mu_t x))|_{t=0}$, which by the Mazur-Orlicz polarization formula (\cite{Mazur}) $a_1 \dotsm a_k = \frac{1}{k!} \sum_{j=1}^k (-1)^{k-j} \sum_{i_1 < \dotsc < i_j} (a_{i_1} + \dotsc + a_{i_j})^k$ (for any $a_1\dotsc a_k$ in a commutative ring) is given by a linear combination of terms $f(t,\e,x) \coleq (\D_Z^x + c(\pd_t - \kappa \D_X^x))^{\abso{\alpha}+1} (R(\mu_{-t}^* \tilde \phi_{\e,x}, \mu_t x))$ 
at $t=0$ with $Z \in \bN_0^n$, $Z \le \alpha$, $c \in \{0,1\}$, $(Z,c) \ne (0,0)$, 
for which is hence suffices to verify the growth conditions. Assuming the contrary, $\exists K, \alpha$ (a) $\forall N, q$ (b) $\exists m_0$ $\forall q$; $\exists \tilde \phi \in \hlsk q \Omega$ $\exists (\e_j)_j \searrow 0$, $\e_j < 1/j$, $\exists (x_j)_j \in K^\bN$: $\abso{f(0,\e_j, x_j)} > j \cdot \e_j^{-N}$ or $> j \cdot \e_j^m$, respectively, $\forall j$. By assumption on $R$ one knows that (a) $\exists q_0, N_0$ (b) $\exists q_0$; $\forall \tilde \psi \in \hlsk{q_0}{\Omega}$: $\sup_{x \in K}\abso{(\D_Z^x + c(\pd_{x_i} - \kappa \pd_t))^{\abso{\alpha}+1} (R(\beta_{-t}^* \tilde \psi_{\e,x}, \beta_t x))} = O(\e^{-N_0})$ or $O(\e^m)$, respectively, where $\beta$ is the flow of $\kappa e_i$. Set $N=N_0$, $q = q_0$ above. Using the chain rule (\cite{hardycomb}), $f(t,\e,x)$ is given by
\begin{multline*}
 \sum_{\substack{\pi_1,\pi_2\\k_1+k_2 = \abso{\alpha}+1}} \binom{\abso{\alpha}+1}{k_1} (\ud_1^{\abso{\pi_1}}\ud_2^{\abso{\pi_2}} R)(\mu^*_{-t}\tilde \phi_{\e,x}, \mu_t x) \cdot \\
 \prod_{B_1 \in \pi_1} (\D_Z^x + c (\pd_t - \kappa \D_X^x))^{\abso{B_1}} ( \mu_{-t}^* \tilde \phi_{\e,x})
\cdot
 \prod_{B_2 \in \pi_2} (\D_Z^x + c (\pd_t - \kappa \D_X^x))^{\abso{B_2}} ( \mu_t x),
\end{multline*}
where $\pi_j$ runs through all partitions of $\{1,\dotsc,k_j\}$, $\abso{\pi_j}$ is the number of blocks in $\pi_j$, and the products run through all blocks of the respective partition. Applying the chain rule in the same way to $(\D_Z^x + c(\pd_{x_i} - \kappa \pd_t))^{\abso{\alpha}+1}(R(\beta_{-t}^* \tilde \psi_{\e,x}, \beta_t x))$, one sees that this expression is equal to $f(t,\e,x)$ if $\forall k=0,\dotsc,\abso{\alpha}+1$
\begin{align}
(\D_Z^x + c (\pd_t - \kappa \D_X^x))^k(\mu_{-t}^* \tilde \phi_{\e,x}) &= (\D_Z^x + c(\pd_{x_i} + \kappa \pd_{y_i}))^k \tilde \psi_{\e,x} \label{fluesse1} \\
(\D_Z^x + c(\pd_t - \kappa \D_X^x))^k(\mu_t x) & = (\D_Z^x + c(\pd_{x_i} - \kappa\pd_t))^k\beta_t x \label{fluesse2}
\end{align}
With
$\tilde \phi_\beta = \pd_{x+y}^{\beta_i} (((Z^i + 1 - c)\pd_{x_i+y_i} - \kappa c (\D_X^x + \Lie_X^y))/(Z^i + 1))^{\beta_i} \tilde \phi$ for $\abso{\beta} \le \abso{\alpha}+1$ define $\tilde \psi$ as in (LSK7).
A short calculation shows that $(\pd_x^{\gamma - \gamma_i e_i} (Z^i \pd_{x_i} + c(\pd_{x_i} + \kappa \pd_{y_i}))^{\gamma_i} \tilde \psi)_{\e_j, x_j} = (\pd_x^{\gamma - \gamma_i e_i} (Z^i \pd_{x_i} - c (\Lie_X^y + \kappa \D_X^x))\tilde \phi_{\e_j,x_j}$ and thus \eqref{fluesse1} holds at $(\e, x) = (\e_j, x_j)$ $\forall j$ for $\abso{\gamma} \le k$. Equation \eqref{fluesse2} holds trivially at $(t,x)=(0, x_0)$ if $\kappa c X(x_0)=0$. Otherwise, by the rectification theorem there is a local diffeomorphism $\rho$ and a vector $v \in \bR^n$ such that $\D\rho(x)X(x) = v \in \bR^n$ and $\mu(t,x) = \rho^{-1}(\rho(x) + t v)$ for $(t,x)$ in a neighborhood of $(0, x_0)$, which implies $((\D_X^x)^k\pd_t^l \mu)(t,x) = d^{k+l}(\rho^{-1})(\rho(x) + t v)\cdot v^{k+l}$ and thus $(\pd_t - \D_X^x)^k \mu_t x = 0 = (\pd_{x_i} - \pd_t)^k \beta_t x$. In sum this gives a contradiction to our assumption.
\end{proof}

\section{Association}

No discussion of Colombeau algebras would be complete without mention of the concept of association, which provides a means to interpret nonlinear generalized functions in the context of linear distribution theory. We give some elementary results here which are typical for all Colombeau algebras and easily obtained by help of Proposition \ref{additional}.

\begin{definition}\label{assocdef}$R, S \in \hEm \Omega$ are called associated with each other, written $R \approx S$, if $\forall \psi \in \ccD(\Omega)$ $\exists q \in \bN$ $\forall \tilde \phi \in \hlsk q\Omega$: $(R-S)(\tilde \phi_{\e,x}, x)$ converges, as a function in $x$, to $0$ in $\ccD'(\Omega)$ for $\e \to 0$.
\end{definition}

Because a negligible function evidently is associated with zero this definition is independent of the representatives and we may talk of association of elements of $\hG \Omega$. The following classical results are immediate consequences of (LSK1) and (LSK5):

\begin{proposition}
 \begin{enumerate}[label=(\roman*)]
  \item For $f \in C^\infty(\Omega)$ and $u \in \ccD'(\Omega)$, $\iota(f)\iota(u) \approx \iota(f u)$.
  \item For $f,g \in C(\Omega)$, $\iota(f) \iota(g) \approx \iota(fg)$.
  \end{enumerate}
\end{proposition}
\begin{proof}
(i) $\langle f(x) \langle u, \tilde \phi_{\e,x} \rangle - \langle fu, \tilde \phi_{\e,x} \rangle , \psi(x) \rangle \to 0$ for all $\tilde \phi \in \hlsk 0 \Omega$ by (LSK5).
(ii) For $f,g \in C(\Omega)$ and $\tilde \phi \in \hlsk 0 \Omega$, with $C$ from (LSK1) we can for small $\e$ estimate the modulus of $\int_{\ball{C\e}{x}} f(y) (g(x) - g(y)) \tilde \phi_{\e,x}(y)\,\ud y$ uniformly for $x$ in compact sets by
\begin{equation}\label{gehtgegennull}
\sup_{y \in \ball{C\e}{x}} \abso{f(y)(g(x) - g(y))} \cdot C_\e \cdot \sup_{y \in \Omega}\abso{\tilde \phi_{\e,x}(y)} \to 0
\end{equation}
 where $C_\e = O(\e^n)$ is the volume of $\ball{C\e}{x}$. In particular this holds for $f=1$, so uniformly on compact sets we have $\langle g, \tilde \phi_{\e,x} \rangle - g(x) \to 0$ and boundedness of $\langle g, \tilde \phi_{\e,x}\rangle$. It follows that for $f,g \in C(\Omega)$, $\langle f, \tilde \phi_{\e,x} \rangle \cdot \langle g, \tilde \phi_{\e,x} \rangle - \langle fg, \tilde \phi_{\e,x} \rangle$, which equals $\langle f, \tilde \phi_{\e,x} \rangle ( \langle g, \tilde \phi_{\e,x} \rangle - g(x) ) + \langle f(y)(g(x) - g(y)), \tilde \phi_{\e,x}(y) \rangle$, converges to zero uniformly for $x$ in compact sets and thus weakly in $\ccD'(\Omega)$.
\end{proof}

Being associated is a local property:

\begin{lemma}\label{assoclocal}Given $R,S \in \hEm \Omega$, if $R$ and $S$ are associated with each other then their restrictions to every open subset of $\Omega$ are so. Conversely, if their restrictions to all elements of an open cover of $\Omega$ are associated with each other, then so are $R$ and $S$.
\end{lemma}
\begin{proof}
The first part is clear using (LSK4): for $U \subseteq \Omega$ open and $\psi \in \ccD(U)$, Definition \ref{assocdef} gives some $q$ such that for $\tilde \phi \in \hlsk q \Omega$, $\langle (R-S)(\tilde \phi_{\e,x}, x), \psi(x)\rangle \to 0$; for any $\tilde \psi \in \hlsk q U$ then there exists $\tilde \phi \in \hlsk q \Omega$ such that $\tilde \psi_{\e,x} = \tilde \phi_{\e,x}$ for $x \in \supp \psi$ and small $\e$, thus $\langle (R|_U - S|_U)(\tilde \psi_{\e,x}, x), \psi(x) \rangle = \langle (R-S)(\tilde \phi_{\e,x}, x), \psi(x) \rangle \to 0$.

For the second part, let $\psi \in \ccD(\Omega)$ and an open cover $(U_\alpha)_\alpha$ of $\Omega$ be given. Choose a subordinate partition of unity $(\chi_j)_j$. With $\psi_j \coleq \chi_j \cdot \psi$ we then can write $\psi = \sum \psi_j$ for finitely many $j$ which we enumerate as $1,2,\dotsc,m$ for some $m \in \bN$; furthermore, $\supp \psi_j \subseteq U_{\alpha(j)}$ for some $\alpha(j)$.

For each $j=1\dotsc m$ by assumption there exists $q_j$ such that for all $\tilde \phi_j \in \hlsk{q_j}{U_j}$, $\langle (R-S)|_{U_j}((\tilde \phi_j)_{\e,x}, x), \psi_j(x) \rangle \to 0$. With $q = \max q_j$ and $\tilde \phi \in \hlsk q \Omega$, $\langle (R-S)(\tilde \phi_{\e,x}, x), \psi(x) \rangle = \sum_{j=1}^m \langle (R - S)(\tilde \phi_{\e,x}, x), \psi_j(x) \rangle$ equals (using (LSK1)) $\sum_{j=1}^m \langle (R-S)|_{U_i}(\tilde \phi_{\e,x}, x), \psi_j(x) \rangle$.

 For each $j$ we can by (LSK4) replace $\tilde \phi$ by $\tilde \phi_j \in \hlsk q {U_j} \subseteq \hlsk {q_j} {U_j}$ such that $\tilde \phi_{\e,x} = (\tilde \phi_j)_{\e, x}$ for all $x \in \supp \psi_j$ and $\e \le \e_0$, from which the claim follows.
\end{proof}

%

\section{Smoothing kernels}\label{sec_testobj}

We use the following Lemma (\cite[Lemma 10.1]{found}).

\begin{lemma}\label{lambdalemma} Let $1 > \e_1 > \e_2 > \dotsc \to 0$, $\e_0=2$. Then there exist $\lambda_j \in \ccD(\bR)$ ($j=1,2,\dotsc$) having the following properties: 1) $\supp \lambda_j = [\e_{j+1}, \e_{j-1}]$, 2) $\lambda_j(x) > 0$ for $x \in (\e_{j+1}, \e_{j-1})$, 3) $\sum_{j=1}^\infty \lambda_j(x) = 1$ for $x \in I$, 4) $\lambda_j(\e_j) = 1$ and 5) $\lambda_1(x) = 1$\quad for $x \in [\e_1, 1]$. 
\end{lemma}

\begin{proposition}\label{notempty}
 $\hlsk q \Omega$ is not empty.
\end{proposition}
\begin{proof}
 In case of $\Omega = \bR^n$ we define the prototypical smoothing kernel $\tilde\phi^\circ \in \Cinf(I \times \bR^n, \ccD(\bR^n))$ by $\tilde \phi^\circ_{\e,x}(y) \coleq \e^{-n} \varphi\left((y-x)/\e\right)$ where $\varphi \in \ccD(\bR^n)$ has integral $1$ and vanishing moments of order up to $q$. We verify the conditions of Definition \ref{skdef}: (LSK1) follows from $\supp \varphi((.-x)/\e) = \e \supp \varphi + x$, (LSK2) is clear. For (LSK3), $\int f(y)(\pd_x^\alpha \tilde \phi^\circ)_{\e,x}(y)\,\ud y = \int (\pd^\alpha f)(y) \e^{-n} \varphi((y-x)/\e)\,\ud y = \int (\pd^\alpha f)(x+\e z)\varphi(z)\,\ud z = f^{(\alpha)}(x) + O(\e^{q+1})$ is then obtained by Taylor expansion of $f$ at the point $x$ because $\varphi$ has vanishing moments up to order $q$. Hence, $\tilde\phi^\circ \in \hlsk q {\bR^n}$.

In the general case of an open subset $\Omega \subseteq \bR^n$ we choose an increasing sequence $(K_j)_{j \in \bN}$ of compact sets $K_1 \csub K_2 \csub \dotsc $ whose union is $\Omega$ 
and functions $\chi_j \in \ccD(\bR^n)$ such that $\chi_j \equiv 1$ on $K_j$ and $\supp \chi_j \subseteq K_{j+1}$. 
Let $1>\e_1>\e_2>\dotsc \to 0$, $\e_0=2$ and choose a partition of unity $(\lambda_j)_{j \in \bN}$ on $I$ as in Lemma \ref{lambdalemma}. Define $\tilde \phi \in \Cinf(I \times \Omega, \ccD(\Omega))$ by $\tilde\phi_{\e,x}(y) \coleq \sum_j \lambda_j(\e) \chi_j(y) \tilde\phi^\circ_{\e,x}(y)$ for $\e \in I$ and $x,y\in\Omega$.
Then $\tilde \phi$ satisfies the conditions of Definition \ref{skdef} because for each $K \csub \Omega$ 
the equality $\tilde \phi_{\e,x} = \tilde \phi^\circ_{\e,x}$ holds for small $\e$ and $x \in K$.
\end{proof}

For the subsequent proofs we recall the multivariate chain rule from \cite{Constantine} in our notation.
\begin{proposition}\label{chainrule}
Let $d,m \in \bN$, $g = (g_1,\dotsc,g_m)\colon U \subseteq \bR^d \to \bR^m$, $f\colon V \subseteq \bR^m \to \bC$ where $U$ and $V$ are open, and $x_0 \in U$ be given with $g(x_0) \in V$. Let $0 \ne \alpha \in \bN_0^n$ be given. Assuming $g \in C^\alpha(U)$ and $f \in C^{\abso{\alpha}}(V)$,
\[ \pd^\alpha (f \circ g)(x) = \sum_{1 \le \abso{\beta} \le \abso{\alpha}} (\pd^\beta f)(g(x)) \sum_{p(\alpha, \beta)} (\alpha!) \prod_{j=1}^{\abso{\alpha}} \frac{(\pd^{l_j}g)^{k_j}(x)}{k_j!(l_j!)^{\abso{k_j}}} \]
for $x \in U$, where $p(\alpha, \beta)$ consists of all $(k_1,\dotsc,k_{\abso{\alpha}}; l_1,\dotsc,l_{\abso{\alpha}}) \in (\bN_0^m)^{\abso{\alpha}} \times (\bN_0^d)^{\abso{\alpha}}$ such that for some $1 \le s \le \abso{\alpha}$, $k_i=0$ and $l_i=0$ for $1 \le i \le \abso{\alpha}-s$; $\abso{k_i} > 0$ for $\abso{\alpha}-s+1\le i\le {\abso{\alpha}}$; and $0 \prec l_{\abso{\alpha}-s+1} \prec \dotsm \prec l_{\abso{\alpha}}$ are such that $\sum_{i=1}^{\abso{\alpha}} k_i=\beta$, $\sum_{i=1}^{\abso{\alpha}} \abso{k_i}l_i = \alpha$. Here $\pd^{l_j} g = (\pd^{l_j} g_1,\dotsc,\pd^{l_j} g_m)$ and $\alpha \prec \beta$ means that either $\abso{\alpha} < \abso{\beta}$ or for some $k < n$, $\alpha_i = \beta_i$ for $i \le k$ and $\alpha_{k+1} < \beta_{k+1}$.
\end{proposition}

\begin{proposition}\label{skdiffeoinv}Given $\tilde \phi \in \hlsk q {\Omega'}$ and a diffeomorphism $\mu\colon \Omega \to \Omega'$, $\mu^*\tilde \phi \in \hlsk q \Omega$.
\end{proposition}
\begin{proof}
We verify the conditions of Definition \ref{skdef}. Set $\tilde \psi \coleq \mu^*\tilde \phi$. First, (LSK1) follows because $\mu$ is locally Lipschitz continuous. For (LSK2) we have to estimate derivatives of $\tilde \phi_{\e, \mu x}(\mu y) \cdot \abso{\det \D\mu(y)}$. We write $\tilde \phi_\e(x,y) = \tilde \phi_{\e,x}(y)$, justified by the exponential law \cite[3.12]{KM}, and define the bijective map $T(x,y) \coleq (x, y-x)$. Because $\abso{\det \D\mu(y)}$ does not depend on $\e$ and $y$ effectively only ranges over a compact set because of (LSK1), it suffices to estimate derivatives of $\tilde \phi_{\e, \mu x}(\mu y)$; assuming $(\alpha,\beta) \ne (0,0)$ (otherwise the case is trivial) we write $\pd_{x+y}^\alpha\pd_{y}^\beta (\tilde \phi_\e(\mu x, \mu y))$ as
$\pd_{(x,y)}^{(\alpha,\beta)} ( (\tilde \phi_\e \circ T^{-1}) \circ (T \circ (\mu \times \mu) \circ T^{-1}))(T(x,y))$
for $x$ in a compact set $K \csub \Omega$ and $y \in \Omega$.
Note that $\tilde \phi_\e \circ T^{-1}$ is smooth at $T(\mu(x), \mu(y))$ and $T \circ (\mu \times \mu) \circ T^{-1}$ is smooth at $T(x,y)$. By the chain rule that expression is equal to 
\begin{multline}\label{wurst1}
\sum_{1 \le \abso{(\alpha', \beta')} \le \abso{(\alpha,\beta)}} \Biggl( (\pd_{(x,y)}^{(\alpha', \beta')} (\tilde \phi_\e \circ T^{-1}))(T(\mu x, \mu y)) \cdot \\
\sum_{p((\alpha,\beta), (\alpha', \beta'))} (\alpha,\beta)! \prod_{j=1}^{\abso{(\alpha,\beta)}} \frac{(\pd^{l_j}g)^{k_j}(T(x,y))}{(k_j!)(l_j!)^{\abso{k_j}}} \Biggr)
\end{multline}
where $g \coleq T \circ (\mu \times \mu) \circ T^{-1}$
and $p((\alpha,\beta),(\alpha', \beta'))$ consists of tuples $(k_1,\dotsc;l_1,\dotsc)$ satisfying $\sum k_i=(\alpha', \beta')$ and $\sum \abso{k_i}l_i = (\alpha,\beta)$.
Noting that
\[ (\pd_{(x,y)}^{(\alpha', \beta')} ( \tilde \phi_\e \circ T^{-1}))(T(\mu x, \mu y)) = (\pd_{x+y}^{\alpha'}\pd_{y}^{\beta'}\tilde \phi)_{\e, \mu x}(\mu y) \]
we see by (LSK2) that this factor in \eqref{wurst1} is $O(\e^{-n-\abso{\beta'}})$. Because $\abso{\beta'}$ can be as large as $\abso{(\alpha, \beta)}$ this growth has to be compensated for by the remaining factors. 
Now $(\pd^{l_j}g)^{k_j}(T(x,y))$ with $l_j = (l_j^{(1)}, l_j^{(2)})$ and $k_j=(k_j^{(1)}, k_j^{(2)})$ is given by (with $0^0 \coleq 1$) $(\pd^{l_j^{(1)}}\mu)^{k_j^{(1)}}(x) \cdot ((\pd^{l_j^{(1)}}\mu)(y) - (\pd^{l_j^{(1)}}\mu)(x))^{k_j^{(2)}}$ if $l_j^{(2)} = 0$, and $0^{k_j^{(1)}} \cdot ((\pd^{l_j^{(1)} + l_j^{(2)}} \mu)(y))^{k_j^{(2)}}$ if $l_j^{(2)} \ne 0$.

From this, (LSK1) and Lipschitz continuity of derivatives of $\mu$ one gains that $(\pd^{l_j}g)^{k_j}(T(x,y))$ is $O(\e^{\abso{k_j^{(2)}}})$ if $l_j^{(2)} = 0$ and $O(1)$ if $l_j^{(2)} \ne 0$, so the $\prod_j$ in \eqref{wurst1} gives $O(\e^{m})$ with $m = \sum_{j} \abso{k_j^{(2)}}  - \sum_{j: l_j^{(2)} \ne 0} \abso{k_j^{(2)}} \ge \abso{ \beta' } - \abso{\sum_{j} \abso{k_j^{(2)}} \cdot l_j^{(2)}} \ge \abso{\beta'} - \abso{\beta}$
which leaves $O(\e^{-n-\beta})$ for the growth of \eqref{wurst1} as desired.

For (LSK3), the case of $\alpha=0$ is clear by substitution in the integral.
Otherwise, we have by Proposition \ref{chainrule} that $(\pd_x^\alpha (\mu^*\tilde \phi))_{\e,x}(y) = \pd_x^\alpha (\tilde \phi_{\e, \mu x}(\mu y) \cdot \abso{\det \D\mu(y)})$ is given by
\begin{equation*}
\sum_{1 \le \abso{\beta} \le \abso{\alpha}} (\pd_{x}^{\beta} \tilde \phi)_{\e,\mu x}(\mu y) \cdot \abso{\det \D\mu(y)} \sum_{p(\alpha, \beta)} \alpha! \prod_{j=1}^{\abso{\alpha}} \frac{(\pd^{l_j}\mu)^{k_j}(x)}{k_j!(l_j!)^{\abso{k_j}}}
\end{equation*}
where $p(\alpha,\beta) = (k_1,\dotsc,k_{\abso{\alpha}}; l_1,\dotsc,l_{\abso{\alpha}})$.
When integrating the product of this with $f(y)$, substitution gives
\begin{align*}
&\sum_{1 \le \abso{\beta} \le \abso{\alpha}}\int_\Omega  f(y)  (\pd_{x}^{\beta} \tilde \phi)_{\e,\mu x}(\mu y) \abso{\det \D\mu(y)} \,\ud y \cdot \sum_{p(\alpha, \beta)} \alpha! \prod_{j=1}^{\abso{\alpha}} \frac{(\pd^{l_j}\mu)^{k_j}(x)}{k_j!(l_j!)^{\abso{k_j}}}\\
= &\sum_{1 \le \abso{\beta} \le \abso{\alpha}} \underbrace{\int (f \circ \mu^{-1})(y)  (\pd_{x}^{\beta} \tilde \phi)_{\e,\mu x}(y) \,\ud y}_{=\pd_{x}^{\beta} (f \circ \mu^{-1})(\mu(x)) + O(\e^{q+1})} \cdot \sum_{p(\alpha, \alpha')} \alpha! \prod_{j=1}^{\abso{\alpha}} \frac{(\pd^{l_j}\mu)^{k_j}(x)}{k_j!(l_j!)^{\abso{k_j}}} \\
= &\ ((f \circ \mu^{-1}) \circ \mu)^{(\alpha)}(x) + O(\e^{q+1}) = (\pd^\alpha f)(x) + O(\e^{q+1})
\end{align*}
uniformly for $x$ in compact sets, which is the desired result.
\end{proof}

We will now show (LSK1-D) for the smoothing kernels of Definition \ref{skdef} and thus establish Proposition \ref{additional}.

%

\begin{proof}[Proof of Proposition \ref{additional}]
(LSK4): Let $U,V$ be open subsets of $\Omega$, $K \csub U \cap V$, $q \in \bN_0$ and $\tilde\phi \in \hlsk q U$. Choose $\chi \in \ccD(U \cap V)$ with $\chi \equiv 1$ on $K$. Let $\e_0 \in I$ be such that $\supp \tilde\psi_{\e,x} \subseteq U \cap V$ for $x \in \supp \chi$ and $\e \le \e_0$ and fix any $\lambda \in \Cinf(I)$ which is $1$ on $(0, \e_0/2)$ and $0$ on $[\e_0, 1]$. Fix an arbitary smoothing kernel $\tilde \psi^\circ \in \hlsk q V$ and define $\tilde\psi_{\e,x} \coleq \lambda(\e) \chi(x) \tilde\phi_{\e,x} + (1-\lambda(\e)\chi(x))\tilde \psi^\circ_{\e,x}$.
Then $\tilde \psi \in \hlsk q V$: any given $L \csub V$ can be decomposed as $L = L_1 \cup L_2$ with $L_1 \csub U \cap V$ and $L_2 \csub V \setminus \supp \chi$; for $\e \le \e_0/2$, (LSK1), (LSK2) and (LSK3) then are easily seen to be satisfied on $L_1$ and $L_2$. For $\e \le \e_0/2$ and $x \in K$, finally, $\tilde \psi_{\e,x} = \tilde \phi_{\e,x}$.

(LSK5): Let $u \in \ccD'(\Omega)$, $k \in \bN_0$, $X_1, \dotsc, X_k \in C^\infty(\Omega, \bR^n)$ and $\varphi \in \ccD(\Omega)$. By (LSK1) $\supp \tilde \phi_{\e,x}$ is contained, for small $\e$, in a relatively compact open neighborhood $U$ in $\Omega$ of $\supp \varphi$ for all $x \in \supp \varphi$. By the structure theorem for distributions we can write $u|_U = (-1)^{\abso{\beta}}\pd^\beta f|_U$ for a continuous function $f$ with support in an arbitrarily small neighborhood of $\overline{U}$, so 
%
%
%
$\langle \langle u, \Lie_{X_1}^x \dotsc \Lie_{X_k}^x \tilde\phi_{\e,x}\rangle , \varphi(x) \rangle$ is given by
\begin{align*}
 \langle \langle f(y), &(\pd_y^\beta \tilde \phi)_{\e,x}(y) \rangle, (-1)^k\Lie_{X_1} \dotsc \Lie_{X_k} \varphi(x) \rangle\\
&= \langle \langle f(y), ((\pd_{x+y} - \pd_x)^\beta \tilde \phi)_{\e,x}(y) \rangle, (-1)^k \Lie_{X_1} \dotsc \Lie_{X_k}\varphi(x) \rangle\\
&= \sum_{\beta' \le \beta} \binom{\beta}{\beta'} \langle \langle f(y), (\pd_{x+y}^{\beta'} (-\pd_{x})^{\beta - \beta'}\tilde \phi)_{\e,x}(y)\rangle, (-1)^k \Lie_{X_1} \dotsc \Lie_{X_k} \varphi(x) \rangle \\
&= \sum_{\beta' \le \beta} \binom{\beta}{\beta'} \langle \langle f(y), (\pd_{x+y}^{\beta'} \tilde \phi)_{\e,x}(y) \rangle, (-1)^k\pd^{\beta-\beta'}\Lie_{X_1} \dotsc \Lie_{X_k} \varphi(x) \rangle \\
&= \sum_{\beta' \le \beta} \binom{\beta}{\beta'} \Bigl( \langle \langle f(y) - f(x), (\pd_{x+y}^{\beta'} \tilde \phi)_{\e,x}(y) \rangle, (-1)^k \pd^{\beta-\beta'}\Lie_{X_1} \dotsc \Lie_{X_k}\varphi(x) \rangle \\
&\qquad + \langle f(x) \langle 1, (\pd_{x+y}^{\beta'} \tilde \phi)_{\e,x}(y) \rangle, (-1)^k \pd^{\beta-\beta'}\Lie_{X_1} \dotsc \Lie_{X_k} \varphi(x) \rangle \Bigr)
\end{align*}
Because $f(y) - f(x) \to 0$ uniformly for $x \in \supp \varphi$, $y \in B_{C\e(x)}$ (with $C$ from (LSK1)) and $\e \to 0$ and because $\pd_{x+y}^{\beta'}\tilde\phi_{\e,x}(y)$ is bounded as in (LSK2) the first part of the last sum converges to $0$ similarly as in \eqref{gehtgegennull}. By (LSK3) the limit of the second part is $\langle f(x), (-1)^k \pd^\beta \Lie_{X_1} \dotsc \Lie_{X_k} \varphi(x) \rangle = \langle \Lie_{X_1} \dotsc \Lie_{X_k} u, \varphi \rangle$.

(LSK6) was shown in Proposition \ref{skdiffeoinv}.

(LSK7): 
(LSK1) for $\tilde \psi$ is obvious.

(LSK2) for $\tilde \psi$: For $\alpha \le \delta$ (otherwise the expression is $0$) the derivative $(\pd_{x+y}^\alpha \pd_{y}^\beta \tilde \psi)_{\e,x}(y)$ is given by
\begin{multline*}
     \sum_{j=1}^\infty \lambda_j(\e) \sum_{\alpha \le \delta' \le \delta} \Biggl(\frac{(x-x_j)^{\delta' - \alpha}}{(\delta' - \alpha)!} \left(\frac{\e_j}{\e}\right)^{n+\abso{\beta}} \pd_{y}^{\beta}(\tilde\phi_{\delta'})_{\e_j, x_j}\left(\e_j \frac{y-x}{\e} + x_j\right)\Biggr)
    \end{multline*}
By (LSK2) this can be estimated uniformly for $x \in K$ by $\sum_j \lambda_j(\e) C (\e_j/\e)^{n+\abso{\beta}} \e_j^{-n-\abso{\beta}} = \sum_j \lambda_j(\e) C \e^{-n-\abso{\beta}} = O(\e^{-n-\abso{\beta}})$ for some constant $C>0$.

(LSK3) for $\tilde \psi$ is equivalent to
$\int f(y) (\pd_{x+y}^\alpha \tilde \psi)_{\e,x}(y)\,\ud y = \pd^\alpha(f(x)) + O(\e^{q+1})$ for $\alpha \le \delta$. Note that $\pd^\alpha(f(x))$ means the derivative of the constant $f(x)$, which is zero for $\alpha \ne 0$. The integral is (for $\e \le \e_0$ with $C,\e_0$ from (LSK1))
\[
 \sum_{j=1}^\infty \lambda_j(\e) \sum_{\alpha \le \delta' \le \delta} \frac{(x-x_j)^{\delta' - \alpha}}{(\delta' - \alpha)!} \int_{\ball{C\e}x} f(y) (\tilde \phi_{\delta'})_{\e_j, x_j} (\e_j \frac{y-x}{\e} + x_j)\,\ud y.
\]
Substituting $u = \e_j(y-x)/\e + x_j$ and forming the Taylor expansion of 
$f(\e(u-x_j)/\e_j +x)$ of order $q$ about $x$,
$\int f(y) (\pd_{x+y}^\alpha \tilde \psi)_{\e,x}(y)\,\ud y - \pd^\alpha (f(x))$ without the remainder term is given by
\begin{multline}\label{schaetz1}
 \sum_{j=1}^\infty \sum_{\alpha \le \delta' \le \delta} \sum_{\abso{\gamma} \le q} \lambda_j(\e) \frac{(x-x_j)^{\delta' - \alpha}}{(\delta' - \alpha)!} \left(\frac{\e_j}{\e}\right)^{-\abso{\gamma}} \frac{f^{(\gamma)}(x)}{\gamma!} \cdot \\
\Biggl(\int_{\ball{C\e_j}{x_j}} (u-x_j)^\gamma (\tilde \phi_{\delta'})_{\e_j, x_j} (u)\,\ud u - \pd^{\gamma +\delta'}1\Biggr).
\end{multline}
The term in parantheses is $O(\e_j^{q+1})$ so \eqref{schaetz1} can be estimated uniformly for $x \in K$ by
$\sum_{j=1}^\infty \sum_{\abso{\gamma} \le q}\lambda_j(\e) (\e_j/\e)^{- \abso{\gamma} } O(\e_j^{q+1}) = O(\e^{q+1})$. The remainder is
\begin{multline*}
 \sum_{j=1}^\infty \sum_{\alpha \le \delta' \le \delta} \sum_{\abso{\gamma} = q+1} \lambda_j(\e) \frac{(x-x_j)^{\delta' - \alpha}}{(\delta' - \alpha)!} \frac{q+1}{\gamma!} \e^{q+1}\cdot \\
\int_{\ball{C\e_j}{x_j}} \int_0^1 (1-s)^q (\pd^\gamma f)(x + s\e(u-x_j)/\e_j)\, \ud s \, \left(\frac{u - x_j}{\e_j}\right)^\gamma (\tilde \phi_{\delta'})_{\e_j, x_j} (u)\,\ud u.
\end{multline*}
The double integral is bounded uniformly for $x \in K$, so $O(\e^{q+1})$ remains.
\end{proof}

\section{Global Theory}

We will now extend the construction to manifolds. This requires little more than the right definitions, with which all properties follow effortlessly from the local case.

\begin{definition}\label{mf_basedef}Let $M$ be a manifold.
\begin{enumerate}[label=(\roman*)]
 \item The basic space is $\hE M \coleq \Cinf(\ocM \times M)$. The embeddings $\iota\colon \Dp M \to \hE M$ and $\sigma\colon \Cinf(M) \to \hE M$ are defined as $(\iota u)(\omega, x) = \langle u, \omega \rangle$ for a distribution $u$ and $(\sigma f)(\omega, x) = f(x)$ for a smooth function $f$ on $M$, where $\omega \in \ocM$ and $x \in M$.
\item Let $\mu\colon M \to M'$ be a diffeomorphism from $M$ to another manifold $M'$. Given a generalized function $R \in \hE{M'}$, its pullback $\mu^* R \in \hE M$ is defined as $(\mu^* R)(\omega, x) = R(\mu_*\omega, \mu x)$.
\item The Lie derivative of $R \in \hE M$ with respect to a smooth vector field $X$ on $M$ is defined as $(\hat \Lie_X R)(\omega, x) = - \ud_1 R(\omega, x)(\Lie_X \omega) + (\Lie_X^x R)(\omega, x)$.
\end{enumerate}
\end{definition}

\begin{remark}
By the same reasoning as in the local case $\mu^*R$ and $\hat\Lie_X R$ are smooth; $\hE M$ is an associative, commutative algebra with unit $\sigma(1)\colon(\omega, x) \mapsto 1$, $\iota$ is a linear embedding and $\sigma$ an algebra embedding. As before, pullback and Lie derivatives commute with the embeddings and $\hat \Lie_X$ is only $\bR$-linear in $X$ but not $C^\infty(M)$-linear.
\end{remark}




We use the following notation for the relationship between local and global expressions on a chart $(U, \psi)$:
\begin{enumerate}[label=(\roman*)]
 \item For smooth vector fields, the isomorphism $\fX(U) \cong C^\infty(\psi(U), \bR^n)$ is written as $X \mapsto X_U$ with inverse $Y \mapsto Y^U$.
 \item For $n$-forms, the isomorphism $\Omega^n(U) \cong C^\infty(\psi(U))$ is written as $\omega \mapsto \omega_U$ with inverse $\varphi \mapsto \varphi^U$, where $\omega_U(y) \coleq \varphi_*(\omega)(y)(e_1,\dotsc,e_n)$.
 \item For distributions, the isomorphism $\ccD'(U) \cong \ccD'(\varphi(U))$ is given by $u \mapsto u_U$ with $\langle u_U, \varphi \rangle \coleq \langle u, \varphi^U \rangle$ and its inverse $v \mapsto v^U$, $\langle v^U, \omega \rangle \coleq \langle v, \omega_U \rangle$.
 \item The isomorphism of basic spaces $C^\infty(\Omega^n_c(U) \times U) \cong C^\infty(\ccD(\varphi(U)) \times \varphi(U))$ is given by $R \mapsto R_U$ with $R_U(\varphi, x) \coleq R(\varphi^U, \varphi^{-1} x)$ with inverse $S \mapsto S^U$, $S^U(\omega, x) \coleq S(\omega_U, \varphi x)$.
\end{enumerate}

We then have $(\Lie_X \omega)_U = \Lie_{X_U}(\omega_U)$ and $(\hat \Lie_X R)(\omega, x) = (\hat \Lie_{X_U} R_U)(\omega_U, \varphi x)$. Next we define smoothing kernels on manifolds.

\begin{definition}\label{mf_skdef}A \emph{smoothing kernel of order $q \in \bN_0$} on a manifold $M$ is defined to be a mapping $\Phi \in \Cinf(I \times M, \ocM)$, $(\e,x) \to [ y \to \Phi_{\e,x}(y)]$, satisfying the following conditions for any Riemannian metric $h$ on $M$:
\begin{enumerate}
 \item[(SK1)] $\forall K \csub M$ $\exists \e_0,C>0$ $\forall x \in K$ $\forall \e<\e_0$: $\supp \Phi_{\e,x} \subseteq B^h_{C\e}(x)$,
 \item[(SK2)] $\forall K \csub M$ $\forall j,k \in \bN_0$ $\forall X_1,\dotsc,X_j,Y_1,\dotsc,Y_k \in \fX(M)$:
\[ \norm{(\Lie^{x+y}_{X_1}\dotsm \Lie^{x+y}_{X_j} \Lie^{y}_{Y_1} \dotsm \Lie^{y}_{Y_k} \Phi)_{\e,x}(y)}_h = O(\e^{-n-k}) \]
 uniformly for $x \in K$ and $y \in M$,
 \item[(SK3)] $\forall K \csub M$ $\forall j \in \bN_0$ $\forall X_1,\dotsc,X_j \in \fX(M)$  $\forall f \in \Cinf(M)$:
\[
\int_M f \cdot (\Lie^x_{X_1} \dotsm \Lie^x_{X_j} \Phi)_{\e,x} = (\Lie_{X_1} \dotsm \Lie_{X_j} f)(x) + O(\e^{q+1})
\]
uniformly for $x \in K$.
\end{enumerate}
The space of all smoothing kernels of order $q$ on $M$ is denoted by $\hsk q M$ and is an affine subspace of $\Cinf(I \times M, \ocM)$. The linear subspace parallel to it, denoted by $\hsk {q0} M$, is given by all $\Phi$ satisfying (SK1), (SK2) and the following condition:
\begin{enumerate}
 \item[(SK3')] $\forall K \csub M$ $\forall j \in \bN_0$ $\forall X_1,\dotsc,X_j \in \fX(M)$  $\forall f \in \Cinf(M)$:
\[
\int_M f \cdot (\Lie^x_{X_1} \dotsm \Lie^x_{X_j} \Phi)_{\e,x} = O(\e^{q+1})
\]
uniformly for $x \in K$.
\end{enumerate}

\end{definition}
Note that by \cite[Lemma 3.4]{global} Definition \ref{mf_skdef} does not depend on the choice of the Riemannian metric. Given a chart $(U, \varphi)$ on $M$ we see that smoothing kernels on $U$ correspond exactly to smoothing kernels on $\varphi(U)$ as in Definition \ref{skdef}:

\begin{proposition}\label{skloc} Let $(U,\varphi)$ be a chart on $M$. Then $\hsk q U \cong \hlsk q {\varphi(U)}$ as affine spaces and $\hsk {q0} U \cong \hlsk {q0} {\varphi(U)}$ as linear spaces. 
\end{proposition}
\begin{proof}
The isomorphism is $\tilde \phi_{\e,x} \coleq (\Phi_{\e, \varphi^{-1} x})_U$ with inverse $\Phi_{\e,x} \coleq (\tilde \phi_{\e, \varphi x})^U$.
Taking for $h$ the pullback metric of the Euclidean metric on $\varphi(U)$ to $U$ along $\varphi$, then given $K \csub \varphi(U)$ $\exists \e_0,C$ such that $\supp \Phi_{\e,x} \subseteq B^h_{C\e}(x)$ $\forall \e\le \e_0$ $\forall x \in \varphi^{-1}(K)$ and thus
$\supp \tilde \phi_{\e,x} = \varphi(\supp \Phi_{\e, \varphi^{-1} x}) \subseteq \varphi(B_{C\e}^h(\varphi^{-1}(x))) = B_{C\e}(x)$ $\forall \e\le\e_0$, $x \in K$, thus (SK1) implies (LSK1) for $\tilde \phi$; the converse holds by the same reasoning.

Then, $(\pd_{i_1}^{x+y} \dotsm \pd_{i_k}^{x+y} \pd_{j_1}^y \dotsm \pd_{j_l}^y \tilde \phi)_{\e,x}$ equals $((\Lie_{\pd_{i_1}}^{x+y} \dotsm \Lie_{\pd_{i_k}}^{x+y} \Lie_{\pd_{j_1}}^y \dotsm \Lie_{\pd_{j_l}}^y \Phi)_{\e, \varphi^{-1} x})_U$ which implies that (LSK2) for $\tilde \phi$ is equivalent to (SK2) for $\Phi$, because in (SK2) it obviously suffices to restrict the $X_1,\dotsc,Y_k$ to be elements of $\{\, \pd_1, \dotsc, \pd_n\,\}$.

By the same reasoning, (LSK3) for $\tilde \phi$ is equivalent to (SK3) for $\Phi$ because of
\[ \int_U f \cdot (\Lie_{\pd_{i_1}}^x \dotsm \Lie_{\pd_{i_j}}^x \Phi)_{\e,x} = \int_{\varphi(U)} (f \circ \varphi^{-1})(y) \cdot (\pd_{i_1}^x \dotsm \pd_{i_k}^x  \tilde \phi)_{\e, \varphi x}(y)\,\ud y \]
and similarly for (LSK3') and (SK3').
\end{proof}

Using this isomorphism we also write $\tilde \phi = \Phi_U$ and $\Phi = \tilde \phi^U$, respectively.

\begin{definition}Let $\mu\colon M \to M'$ be a diffeomorphism. Then we define the pullback $\mu^*\Phi$ of a smoothing kernel $\Phi \in \hsk q {M'}$ by $(\mu^*\Phi)_{\e,x} \coleq \mu^*(\Phi_{\e, \mu x})$.
\end{definition}

\begin{proposition}\label{mf_additional}The smoothing kernels of Definition \ref{mf_skdef} satisfy these additional properties:
 \begin{enumerate}
\item[(SK4)]\label{sk4} Let $U,V$ be open subsets of $M$, $K \csub U \cap V$ and $q \in \bN_0$. Given $\Phi \in \hsk q U$ there exist $\e_0>0$ and $\Psi \in \hsk q V$ such that $\Phi_{\e,x} = \Psi_{\e,x}$ for $\e < \e_0$ and $x \in K$.
\item[(SK5)]\label{sk5} $\forall u \in \Dp M$ $\forall \Phi \in \hsk 0 M$ $\forall k \in \bN_0$ $\forall X_1,\dotsc,X_k \in \fX(M)$: $\langle u, \Lie_{X_1}^x \dotsc \Lie_{X_k}^x \Phi_{\e,x} \rangle$ converges (weakly) to $\Lie_{X_1} \dotsc \Lie_{X_k} u$ in $\Dp M$.
\item[(SK6)]\label{sk6} If $\mu\colon M \to M'$ is a diffeomorphism and $\Phi' \in \hsk q {M'}$ then $\mu^*\Phi' \in \hlsk q M$.
\end{enumerate}
\end{proposition}
\begin{proof}
(SK4) is proven exactly as in the local case.

(SK5): Let $\omega \in \ocM$ with support in a set $K$; by using a partition of unity we may without limitation of generality assume that $K$ is contained a chart domain $U$. For small $\e$, $\supp \Phi_{\e,x} \subseteq U$ for all $x \in \supp \omega$, thus $\langle \langle u, \Lie_{X_1}^x \dotsc \Lie_{X_k}^x \Phi_{\e,x} \rangle, \omega(x) \rangle$ equals $\langle \langle u_U, \Lie^x_{(X_1)_U} \dotsc \Lie^x_{(X_k)_U} (\Phi_U)_{\e,x} \rangle, \omega_U(x) \rangle$ and converges to $\langle \Lie_{(X_1)_U} \dotsc \Lie_{(X_k)_U} u_U, \omega_U \rangle$ which in turn equals $\langle \Lie_{X_1} \dotsc \Lie_{X_k} u, \omega \rangle$.


(SK6): Fixing $K \csub M$ for verifying (SK1) -- (SK3) for $\mu^*\Phi'$, we may assume that there are charts $(U, \varphi)$ on $M$ and $(U', \varphi')$ on $M'$ such that $K \csub U$ and $\mu(U) = U'$. Given $\Phi' \in \hsk q {M'}$ there exists, by (SK4), a smoothing kernel $\Psi' \in \hsk q {U'}$ such that $\Phi'_{\e,x} = \Psi'_{\e,x}$ for $x \in \mu(K)$ and small $\e$, to which by Proposition \ref{skloc} there corresponds a local smoothing kernel $\tilde \psi' \in \hlsk q {\varphi'(U')}$. The diffeomorphism $\mu' \coleq \varphi' \circ \mu \circ \varphi^{-1}$ from $\varphi(U)$ to $\varphi'(U')$ gives, by (LSK6), a local smoothing kernel $\tilde \phi \coleq \mu'^*\tilde \psi' \in \hlsk q {\varphi(U)}$ to which in turn there corresponds a smoothing kernel $\Phi \in \hsk q U$. Because $(\mu^* \Psi')_{\e,x} = \Phi_{\e,x}$, the result is obtained.
\end{proof}

(LSK7) has no direct equivalent on the manifold. We come to the definition of moderateness and negligibility.

\begin{definition}[\textbf{D3, D4}]\label{mf_modneg}\begin{enumerate}
\item[(i)] $R \in \hE M$ is called \emph{moderate} if $\forall K \csub M$ $\forall j \in \bN_0$ $\exists q \in \bN_0$ $\exists N \in \bN$ $\forall \Phi \in \hsk qM$ $\forall X_1,\dotsc,X_j$ we have the estimate $\Lie^x_{X_1}\dotsm \Lie^x_{X_j} (R(\Phi_{\e,x}, x)) = O(\e^{-N})$ uniformly for $x \in K$.
The set of all moderate elements of $\hE M$ is denoted by $\hEm M$.
 \item[(ii)] $R \in \hE M$ is called \emph{negligible} if $\forall K \csub M$ $\forall j \in \bN_0$ $\forall m \in \bN$ $\exists q \in \bN_0$ $\forall \Phi\in \hsk q M$ $\forall X_1,\dotsc,X_j$ we have the estimate $\Lie^x_{X_1}\dotsm \Lie^x_{X_j} (R(\Phi_{\e,x}, x)) = O(\e^m)$ uniformly for $x \in K$. The set of all negligible elements of $\hE M$ is denoted by $\hN M$.
\end{enumerate}
\end{definition}

\begin{corollary}\label{modnegloc} Let $(U, \varphi)$ be a chart on $M$. Then $R \in \hE U$ is moderate or negligible, respectively, if $R_U \in \hE{\varphi(U)}$ is so.
\end{corollary}
\begin{proof}Using the relation $R(\Phi_{\e,x}, x) = R_U((\Phi_U)_{\e, \varphi x}, \varphi x)$ the claim is immediate from the definitions and Proposition \ref{skloc}.
\end{proof}

Again we can get rid of the derivatives in the test for negligibility.

\begin{corollary}$R \in \hEm M$ is negligible if and only if Definition \ref{mf_modneg} (ii) holds for $j=0$, which is,
$\forall K \csub M$ $\forall m \in \bN$ $\exists q \in \bN_0$ $\forall \Phi\in \hsk q M$ $R(\Phi_{\e,x}, x) = O(\e^m)$ uniformly for $x \in K$.
\end{corollary}

\begin{definition}
 Let $R \in \hE M$ and $M' \subseteq M$ open. Then the restriction $R|_{M'} \in \hE {M'}$ is defined as $R|_{M'}(\omega, x) \coleq R(\omega, x)$ for $\omega \in \Omega^n_c(M') \subseteq \ocM$ and $x \in M'$.
\end{definition}
As in the local case the following is an immediate consequence of (SK4).
\begin{proposition}\label{mf_localize}
\begin{enumerate}
 \item[(i)]Let $M' \subseteq M$ be open and $R \in \hE M$. If $R$ is moderate or negligible, respectively, then so is $R|_{M'}$.
 \item[(ii)] Let $(U_\alpha)_\alpha$ be an open covering of $M$ and $R \in \hE M$. If for all $\alpha$, $R|_{U_\alpha}$ is moderate or negligible, respectively, then so is $R$.
\end{enumerate}
\end{proposition}

%
By Proposition \ref{mf_localize} (i) restriction is well-defined also on the quotient space:

\begin{definition}
Let $\hat T \in \hG M$ and $M' \subseteq M$. Then the restriction $\hat T|_{M'} \in \hG {M'}$ of $\hat T$ to $M'$ is defined as $\hat T|_{M'} \coleq T|_{M'} + \hN{M'}$ where $T \in \hEm M$ is any representative of $\hat T$.
\end{definition}

\begin{proposition}$\hat\cG$ is a fine sheaf.
\end{proposition}
\begin{proof}
The proof of Proposition \ref{itsasheaf} applies with the obvious modifications; additionally, $\hat\cG$ is fine because it is locally fine (\cite{Dowker}).
\end{proof}

\begin{theorem}
(i) $\iota(\ccD'(M)) \subseteq \hEm M$,
(ii) $\sigma(\Cinf(M)) \subseteq \hEm M$,
(iii) $(\iota - \sigma)(\Cinf(M)) \subseteq \hN M$,
(iv) $\iota(\ccD'(M)) \cap \hN M = \{ 0 \}$.
\end{theorem}
\begin{proof}Insteaf of proving this directly we use the local results: for (i), $\iota u$ is moderate if $\iota u|_U = \iota(u|_U) = (\iota(u_U))^U$ is so on each chart domain $U$, which by Corollary \ref{modnegloc} is the case because $\iota(u_U)$ is moderate; similarily for (ii) and (iii). For (iv), $\iota u|_U$ and thus $\iota (u_U)$ are negligible, which implies $u_U=0$ for all chart domains $U$ and thus $u=0$.
\end{proof}

$\hEm M$ is a subalgebra of $\hE M$ and $\hN M$ is an ideal in $\hEm M$, so we can define the algebra of generalized functions on $M$ as the quotient of moderate modulo negligible functions.
\begin{definition}$\hG M \coleq \hEm M / \hN M$.
\end{definition}

\begin{theorem}
 $\hat\Lie_X$ preserves moderateness and negligibility.
\end{theorem}
\begin{proof}Once again using (LSK4) one sees that $(\hat \Lie_X R)|_U$ is moderate or negligible, respectively, if and only if  $\hat \Lie_{X|_U} R|_U$ is so for all chart domains $U$, which by Corollary \ref{modnegloc} is the case if and only if $(\hat \Lie_{X|_U} R|_U)_U = \hat \Lie_{X_U}R_U$ is moderate or negligible, respectively, which holds by Theorem \ref{diffeoinv}.
\end{proof}

\begin{definition}$R, S \in \hEm M$ are called associated with each other, written $R \approx S$, if $\forall \omega \in \Omega^n_c(M)$ $\exists q \in \bN$ $\forall \Phi \in \hsk qM$: $\lim_{\e \to 0}\int (R-S)(\Phi_{\e,x}, x) \omega(x) = 0$.
\end{definition}
This definition is independent of the representatives and extends to $\hG M$ because elements of $\hN M$ are associated with $0$. The notion of association localizes as well:
\begin{lemma}
\begin{enumerate}[label=(\roman*)]
\item Given $R,S \in \hEm M$ and an open cover $M$, $R \approx S$ if and only if $R|_U \approx S|_U$ for all sets $U$ of the cover. In particular, $R \approx S$ implies $R|_U \approx S|_U$ for any open subset $U$ of $M$.
\item Given $R,S \in \hEm U$ for a chart domain $U$, $R \approx S$ if and only if $R_U \approx S_U$.
\end{enumerate}
\end{lemma}
\begin{proof}
(i) is proven exactly as Lemma \ref{assoclocal} while (ii) follows immediately from the definitions.
\end{proof}

As before, we have:
\begin{proposition}
 \begin{enumerate}[label=(\roman*)]
  \item For $f \in C^\infty(\Omega)$ and $u \in \ccD'(\Omega)$, $\iota(f)\iota(u) \approx \iota(f u)$.
  \item For $f,g \in C(\Omega)$, $\iota(f) \iota(g) \approx \iota(fg)$.
  \end{enumerate}
\end{proposition}
\begin{proof}
 (i) $\iota(f)\iota(u) \approx \iota(fu)$ if and only if $\iota(f)|_U\iota(u)|_U \approx \iota(fu)|_U$ for all $U$ of an atlas, which is the case if and only if $\iota(f_U)\iota(u_U) \approx \iota((fu)_U)$; and similarly for (ii).
\end{proof}

This work was supported by projects P20525 and P23714 of the Austrian Science Fund (FWF).

\end{document}